\documentclass[12pt]{article}
\usepackage{amssymb}
\usepackage{amsmath}
\usepackage{theorem}
\usepackage{latexsym}
\usepackage{amscd}
\usepackage[all]{xy}
\theoremheaderfont{\bf} 
\theorembodyfont{\sl}
\newtheorem{theo}{Theorem}[section]
{\theorembodyfont{\rm} \newtheorem{defi}[theo]{Definition}}
{\theorembodyfont{\rm} }
{\theorembodyfont{\rm} }
{\theorembodyfont{\rm} \newtheorem{notation}[theo]{Notation}}
{\theorembodyfont{\rm} }
\newtheorem{prop}[theo]{Proposition}
\newtheorem{cor}[theo]{Corollary}
\newtheorem{lemma}[theo]{Lemma}

\newenvironment{proof}{{\sc Proof:}}{\mbox{}\hfill$\Box$\par}
\newcommand{\eqr}[1]{~\mbox{$(${\rm \ref{#1}}$)$}}

\newcommand{\junk}[1]{}

\newcommand{\N}{{\mathbb N}}
\newcommand{\F}{{\mathbb F}}
\newcommand{\K}{{\mathbb K}}

\newcommand{\C}{{\mathcal C}}

\newcommand{\wt}{{\rm wt}}
\newcommand{\rank}{{\rm rank}\,}

\newcommand{\zwei}[2]{\left[ \begin{array}{c}
                   #1 \\ #2 \end{array} \right]}

\newcounter{abc}

\title{The Existence of Strongly-MDS Convolutional Codes\footnote
   {This work was supported in part by NSF grants DMS-00-72383 and
    CCR-02-05310.}
       }
\author{
  Ryan Hutchinson\\
  {\small Department of Mathematics and Computer Science}\vspace{-2mm}\\
  {\small Bemidji State University}\vspace{-2mm}\\
  {\small {\em e-mail:} rhutchinson@bemidjistate.edu}
  }
\begin{document}
\maketitle
\begin{abstract}
It is known that maximum distance separable and maximum distance
profile convolutional codes exist over large enough finite fields of
any characteristic for all parameters $(n,k,\delta )$.  It has been
conjectured that the same is true for convolutional codes that are
strongly maximum distance separable. Using methods from linear
systems theory, we resolve this conjecture by showing that, over a
large enough finite field of any characteristic, codes which are
simultaneously maximum distance profile and strongly maximum
distance separable exist for all parameters $(n,k,\delta )$.\\

\end{abstract}

\section{Introduction}
In recent literature on convolutional codes, several new classes of
codes with optimal distance properties have been introduced.  These
classes of codes are known as maximum distance separable (MDS)
codes, maximum distance profile (MDP) codes, and strongly MDS (sMDS)
codes.  MDS codes are characterized by the property that they have
the maximum possible free distance for a given choice of code
parameters. sMDS codes are a subclass of MDS codes having the
property that this maximum possible free distance is attained at the
earliest possible encoding step. MDP codes are characterized by the
property that their column distances grow at the maximum possible
rate for a given choice of code parameters.

In~\cite{ro99a1}, it is shown that MDS convolutional codes exist for
all parameters $(n,k,\delta )$ over sufficiently large finite
fields; in~\cite{12}, a similar result is obtained for codes having
the MDP property. In~\cite{gl03r}, sMDS convolutional codes are
introduced and studied, and they are shown to exist for parameters
$(n,k,\delta )$ satisfying $(n-k) \mid \delta$. In addition, it is
conjectured that convolutional codes possessing the MDP and sMDS
properties together exist for all $(n,k,\delta )$. In this work, we
show that this conjecture is correct. The approach used is
systems-theoretic in nature; to obtain the proof, we make use of the
well-known interpretation of a convolutional code as an
input-state-output linear system as well as results from partial
realization theory.

The structure of this paper is as follows.  In Section 2, we review
relevant ideas from the theory of convolutional codes.  We recall as
well a connection between convolutional codes and input-state-output
linear systems that we will use to obtain our results. In Section 3,
we use a linear systems representation of convolutional codes to
give a characterization of the sMDS property.  In Section 4, we use
this characterization to show the existence, for all parameters
$(n,k,\delta )$, of codes possessing both the MDP and sMDS
properties.

\section{Convolutional Codes and Linear Systems}
In this section, we recall some facts about convolutional codes and
their connection with linear systems.  Throughout this paper, $0$
will be understood to be the zero matrix or vector of the
appropriate size.  Let $k$ and $n$ be positive integers with $k<n$,
$p$ a prime number, $\K$ the algebraic closure of the prime field
$\F _p$, and $\F$ a finite subfield of $\K$.
\begin{defi}
A {\em convolutional code} $\mathcal{C}$ of {\em rate} $k/n$ is a
rank-$k$ direct summand of $\F [s]^n$.
\end{defi}
$\mathcal{C}$ is a free $\F [s]$-module and may thus be viewed as
the column space of a full-rank matrix $G(s)\in \F [s]^{n\times k}$,
called a {\em generator matrix} for $\mathcal{C}$. Two full-rank
$n\times k$ matrices $G_1(s)$ and $G_2(s)$ generate the same code if
and only if there exists a unimodular matrix $U(s)\in \F
[s]^{k\times k}$ such that $G_1(s)=G_2(s)U(s)$.

When convenient, we will (at times with a slight abuse of notation)
make use of the fact that $\F [s]^n$ and $\F ^n[s]$ are isomorphic
$\F [s]$-modules and think of codewords as elements of $\F ^n[s]$.
 For example, the columns of a generator matrix $G(s)$ may be thought
of as polynomials with coefficients in $\F ^n$; we refer to the
degrees of these polynomials as the {\em column degrees} of $G(s)$
and denote the degree of the $j$th column by $\delta _j$. The {\em
high-order coefficient matrix} of $G(s)$, $G_{\infty}$, is the
matrix whose $j$th column is the column coefficient of $s^{\delta
_j}$ in the $j$th column of $G(s)$. In general, $G_{\infty}$ need
not have full rank. It is always possible, though, to find a
unimodular matrix $U(s)\in \F [s]^{k\times k}$ such that $G(s)U(s)$
has a full-rank high-order coefficient matrix (see~\cite{17}).  If
$G_{\infty}$ has full rank, then $G(s)$ is called a {\em minimal
generator matrix}.

An important invariant of a convolutional code is its degree,
defined as follows:
\begin{defi}
The {\em degree} $\delta$ of a convolutional code $\mathcal{C}$ is
the maximal degree of a (polynomial) determinant of a $k\times k$
submatrix of a generator matrix of $\mathcal{C}$.
\end{defi}
This definition makes sense, as multiplication by a unimodular
matrix preserves the degrees of such determinants.  We note that, if
$G(s)$ is a minimal generator matrix of $\mathcal{C}$ with column
degrees $\delta _1,\ldots ,\delta _k$, then $\delta =\sum _{j=1}^{k}
\delta _j$.  A code of rate $k/n$ and degree $\delta$ will be
referred to as an $(n,k,\delta)$-code.

We turn next to notions of distance.  We first recall the definition
of Hamming weight:
\begin{defi}
  Let $v\in \F ^n$ and $v(t):=\sum _{t=0}^{d} v_ts^t \in \F ^n[s]$.
  The {\em Hamming weight of $v$}, wt($v$), is the number
  of nonzero components of $v$.  The Hamming weight of $v(t)$ is $\text{wt}(v(t)):=\sum _{t=0}^{d} \text{wt}(v_t)$.
\end{defi}
For the purpose of error control coding, it is important that the
minimum weight among the codewords of a code be as large as
possible.  This leads to the concept of free distance:
\begin{defi}
The {\em free distance} of a convolutional code $\mathcal{C}$ is
\begin{equation*}
d_{free}(\mathcal{C}):=\min \{\wt(v(t))\,
\vline\,\,v(t)\in\mathcal{C} \backslash \{ 0 \}\} .
\end{equation*}
\end{defi}
Column distances also play an important part in what follows. They
measure the minimum possible distance between truncated codewords:
\begin{defi}
Let $\mathcal{C}$ be a convolutional code.  For $j\in \N _{0}$, the
{\em $j$th column distance} of $\mathcal{C}$ is
$$
d_j^c(\mathcal{C}):=\min \left\{\sum_{t=0}^j\wt(v_t)\,\vline\,
v(t)\in \mathcal{C}\,\, \text{and}\,\, v_0\neq 0_{n} \right\},
$$
where $v_j=0_n$ if $j>\text{deg }v(t)$.
\end{defi}
The following result gives upper bounds for the column distances and
the free distance of a convolutional code.
\begin{prop}                  \label{P-dcj.bound}
Let $\mathcal{C}$ be an $(n,k,\delta)$-code.
\begin{enumerate}
\item For every $j\in \N _{0}$,
\[
d_j^c(\mathcal{C})\leq(n-k)(j+1)+1.
\]
If $d_j^c(\mathcal{C})=(n-k)(j+1)+1$ for some $j$, then
$d_i^c(\mathcal{C})=(n-k)(i+1)+1$ when $i\in \{0,\ldots ,j \}$.
\item \begin{equation*}                                \label{G-Singleton}
d_{free}(\mathcal{C})\leq
(n-k)\Big(\Big\lfloor\frac{\delta}{k}\Big\rfloor+1\Big) +\delta+1.
\end{equation*}
\end{enumerate}
\end{prop}
Statement 1 is proved in~\cite{gl03r}, and statement 2 is proved
in~\cite{ro99a1}.  The bound in 2 is called the {\em generalized
Singleton bound}.

Set $L:=\Big\lfloor\frac{\delta}{k}
\Big\rfloor+\Big\lfloor\frac{\delta}{n-k}\Big\rfloor$ and
$M:=\Big\lfloor\frac{\delta}{k}\Big\rfloor+
\Big\lceil\frac{\delta}{n-k}\Big\rceil$.  We are now ready to define
the code properties of interest in this work:
\begin{defi}\label{DistProp}
Let $\mathcal{C}$ be an $(n,k,\delta)$-code.  Then,
\begin{enumerate}
\item $\mathcal{C}$ is called a {\em maximum distance profile code} ({\em MDP
code}) if
$$
d_L^c(\mathcal{C})=(n-k)(L+1)+1.
$$
\item $\mathcal{C}$ is called a {\em maximum distance separable code} ({\em MDS code}) if
\[
d_{free}(\mathcal{C})=(n-k)\Big(\Big\lfloor\frac{\delta}{ k}
\Big\rfloor+1\Big)+\delta+1.
\]
\item $\mathcal{C}$ is called a {\em strongly MDS code} ({\em sMDS code}) if
\[
d_M^c(\mathcal{C})=(n-k)\Big(\Big\lfloor\frac{\delta}{k}
\Big\rfloor+1\Big)+\delta+1.
\]
\end{enumerate}
\end{defi}

Using the fact that no column distance of $\mathcal{C}$ can exceed
the generalized Singleton bound, one can show that $L$ is the
largest possible value of $j$ for which $d_j^c(\mathcal{C})$ can
attain the upper bound in statement 1 of
Proposition~\ref{P-dcj.bound}. If $\mathcal{C}$ is an MDP code,
then, by Proposition~\ref{P-dcj.bound}, $d_i^c(\mathcal{C})$ attains
this upper bound when $i\in \{0,\ldots ,L \}$.  Thus, 1 says that
the column distances of an MDP code are maximal until it is no
longer possible. Similarly, one can show that, if $j<M$, then
$d_j^c(\mathcal{C})<(n-k)\Big(\Big\lfloor\frac{\delta}{k}\Big\rfloor+1\Big)+\delta+1$.
 Thus, 3 says that, for an sMDS code, the sequence
$\{d_j^c(\mathcal{C})\} _{j\geq 0}$ attains the generalized
Singleton bound at the smallest possible value of $j$.

In the second part of this section, we introduce a connection
between convolutional codes and linear systems that we will use to
obtain our results. Background information for this discussion and
applications of ideas from systems theory to the construction of
convolutional codes may be found, for example, in~\cite{Antsaklis,
8,ro96a1, ro99a, 7}.

Let $A \in \K^{\delta
  \times \delta}, \,\,\, B \in \K^{\delta \times k}, \,\,\, C \in
\K^{(n-k) \times \delta}$, and $D \in \K^{(n-k) \times k}$.  Note
that, since the number of entries in the matrices $(A,B,C,D)$ is
finite, these matrices are actually defined over a finite subfield
$\F$ of $\K$. The matrices $(A,B,C,D)$ describe a time-invariant
linear system through the equations
\begin{eqnarray}  \label{iso}
x_{t+1} & = & Ax_t+Bu_t, \nonumber \\  y_t & = & Cx_t+Du_t,\\
x_0&=&0,\nonumber
\end{eqnarray}
where $x_t \in \F ^{\delta}$, $u_t \in \F ^k$, and $y_t \in \F
^{n-k}$ are called the {\em state vector}, {\em input vector}, and
{\em output vector} at time $t$, respectively.  The matrix quadruple
$(A,B,C,D)$ is called a {\em realization} for the system. We recall
the following well-known definition:
\begin{defi}\label{DefB}
$(A,B)$ is called a {\em reachable pair} if
$$
  \rm{rank} \left(\left[\begin{array}{ccccc}
                    B & AB & \cdots & A^{\delta -2}B & A^{\delta -1}B
                 \end{array}
           \right]\right)=\delta.
$$
$(A,C)$ is called an {\em observable pair} if
$$
  \rm{rank} \left(\left[\begin{array}{ccccc}
                      C^T & (CA) ^T & \cdots & (CA^{\delta -2})^T & (CA^{\delta
                      -1})^T
                   \end{array}
             \right] ^{\it T}\right)=\delta.
$$
\end{defi}
If $(A,B)$ is a reachable pair and $(A,C)$ is an observable pair,
then $(A,B,C,D)$ is called a {\em minimal realization}.  In this
case, $\delta$ is called the {\em McMillan degree} of the system. We
denote by $S ^{\delta}_{k,n}$ the set of minimal realizations of
systems over $\K$ having input vectors of size $k$, output vectors
of size $n-k$, and McMillan degree $\delta$.

Let $\{ x_t \} _{t\geq 0}$ be a sequence of vectors in $\F
^{\delta}$ and $\{\binom{y_t}{u_t}\}_{t\geq 0}$, where $y_t \in \F
^{n-k}$ and $u_t \in \F ^k$, a sequence of vectors in $\F ^n$ having
the following properties:
\begin{enumerate}
\item Equations~(\ref{iso}) are satisfied for all $t\in \N_{0}$;
\item There exists a $d\in\N_{0}$ such that $x_{d +1}=0$ and $u_t=0$
for $t\geq d+1$.
\end{enumerate}
These properties guarantee that the sequence
$\{\binom{y_t}{u_t}\}_{t\geq 0}$ has finite weight. We refer to the
truncated sequence $\{\binom{y_t}{u_t} \} _{t=0}^{d}$ as a {\em
finite-weight sequence for $(A,B,C,D)$}.  The following remarks
connect finite-weight sequences and codewords.

Let $(A,B,C,D)\in S _{k,n}^{\delta}$.  The corresponding transfer
function is $T(s) := C(sI - A)^{-1} B + D$. Let $Q^{-1}(s)P(s)$ be a
left coprime factorization of $T(s)$, and set $H(s) := [-Q(s)~
P(s)]$.  Set
$$
y(s) := y_0 s^{d} + y_1 s^{d - 1} + \cdots + y_{d} \in \F ^{n-k}[s]
$$
and
$$
u(s) := u_0 s^{d} + u_1 s^{d - 1} + \cdots + u_{d} \in \F ^k[s],
$$
and use their coefficients to form the vector sequence
$\{\binom{y_t}{u_t} \} _{t=0}^{d}$.  We then have the following
equivalent conditions; see~\cite{8,ro99a,7} for more details:
\begin{enumerate}
\item The set $\{\binom{y_t}{u_t} \} _{t=0}^{d}$ of vectors is a finite-weight sequence for $(A,B,C,D)$.
\item
\begin{equation*}
\left[
\begin{tabular}{lllccccc}
$0$  & $\cdots\cdots\cdots\cdots $ & \multicolumn{1}{l|}{$0$} & $A^d
B$ & $ A^{d -1}B$ & $\cdots $ & $AB$ & $B$ \\ \hline
\multicolumn{3}{c|}{} & $D$ &$0$  &$\cdots$  &$\cdots$  &$0$  \\
\multicolumn{3}{c|}{} & $CB$ & $D$ &$\ddots$  &  &$\vdots$  \\
\multicolumn{3}{c|}{$-I_{(d +1)(n-k)}$} & $CAB$ & $CB$  &$\ddots$  &$\ddots$  &$\vdots$  \\
\multicolumn{3}{c|}{} & $\vdots $ &$\vdots$  & $\ddots$ & $\ddots$ &$0$  \\
&&\multicolumn{1}{l|}{} & $CA^{d -1}B$ & $CA^{d -2}B$ & $ \cdots $ &
$CB$ & $D$
\end{tabular}
\right] \left[
\begin{array}{c}
y_0 \\
y_1 \\
\vdots  \\
y_{d}  \\   \hline
u_0\\
u_1 \\
\vdots  \\
u_{d}
\end{array}
\right] =0.
\end{equation*}
\item There exists a `state vector polynomial'
$$
x(s) = x_0s^{d}+x_{1}s^{{d}-1} + \cdots + x_{d} \in \F ^{\delta}[s]
$$
such that
\begin{equation*}
\label{kern} \left[
\begin{array}{ccc}
sI-A&0&-B\\ -C&I_{n-k}&-D
\end{array}
\right]\left[
\begin{array}{c}
x(s)\\y(s)\\u(s)
\end{array}
\right] =0.
\end{equation*}
\item
$ H(s)\zwei{y(s)}{u(s)}= [-Q(s) \ P(s)] \zwei{y(s)}{u(s)}=0. $
\item $y(s)=T(s)u(s)$.
\end{enumerate}
Further, the right $\F [s]$-kernel of $H(s)$ is an $(n,k,\delta
)$-code $\mathcal{C}$.

The code $\mathcal{C}$ is not quite suitable for our purposes.  This
is due to the fact that the finite-weight sequence
$$
\binom{y_{0}}{u_{0}},\binom{y_{1}}{u_{1}},\ldots ,\binom{y_{d
-1}}{u_{d -1}},\binom{y_{d}}{u_{d}}
$$
corresponds with the codeword
$$
\binom{y_{d}}{u_{d}}+\binom{y_{d -1}}{u_{d -1}}s+\cdots
+\binom{y_{1}}{u_{1}}s^{d -1}+\binom{y_{0}}{u_{0}}s^{d} \in
\mathcal{C}.
$$
Working in the systems setting, we will show there is a realization
$(A,B,C,D)\in S ^{\delta}_{k,n}$ for which any finite-weight
sequence
$$
\binom{y_{0}}{u_{0}},\binom{y_{1}}{u_{1}},\ldots ,\binom{y_{d
-1}}{u_{d -1}},\binom{y_{d}}{u_{d}}
$$
(with $u_0 \neq 0$) formed using (\ref{iso}) has the properties that
$$
\sum_{t=0}^L \text{wt}\left (\binom{y_{t}}{u_{t}}\right )
=(L+1)(n-k)+1
$$
and
$$
\sum_{t=0}^M \text{wt}\left (\binom{y_{t}}{u_{t}}\right )\geq
(n-k)\Big(\Big\lfloor\frac{\delta}{ k} \Big\rfloor+1\Big)+\delta+1.
$$
Due to the order reversal noted above, it will not necessarily be
true that $d_L^c(\mathcal{C})=(L+1)(n-k)+1.$ The next result shows
how to overcome this problem:
\begin{prop}
Let $\mathcal{C}$ be an $(n,k,\delta )$-code with minimal generator
matrix $G(s)$.  Let $\overline{G(s)}$ be the matrix obtained by
replacing each entry $p_{ij}(s)$ of $G(s)$ by
$\overline{p_{ij}(s)}:=s^{\delta _j}p_{ij}(s^{-1})$, where $\delta
_j$ is the $j$th column degree of $G(s)$. Then, $\overline{G(s)}$ is
a minimal generator matrix of an $(n,k,\delta )$-code
$\overline{\mathcal{C}}$, and
$$
\binom{y_{0}}{u_{0}} +\binom{y_{1}}{u_{1}} s + \cdots +
 \binom{y_{d -1}}{u_{d -1}} s^{d
-1} +\binom{y_{d}}{u_{d}} s^{d} \in\overline{\mathcal{C}}
$$
if and only if
$$
\binom{y_{d}}{u_{d}}+\binom{y_{d -1}}{u_{d -1}}s+\cdots
+\binom{y_{1}}{u_{1}}s^{d -1}+\binom{y_{0}}{u_{0}}s^{d} \in
\mathcal{C}.
$$
\end{prop}
\begin{proof}
First, $\overline{G(s)}$ has rank $k$, since a $k\times k$ minor of
$\overline{G(s)}$ is zero if and only if the corresponding minor of
$G(s)$ is.

Next, let $G_0$ denote the $n\times k$ matrix whose $ij$th entry is
$p_{ij}(0)$ and $\overline{G} _0$ the $n\times k$ matrix whose
$ij$th entry is $\overline{p_{ij}(0)}$.  Because $\mathcal{C}$ is a
summand of $\F [s]^n$, $G_0$ has full rank, so that each column of
$G_0$ has at least one nonzero entry.  This means that
$G_0=\overline{G} _{\infty}$, $\overline{G} _{\infty}$ has full
rank, corresponding columns of $G(s)$ and $\overline{G(s)}$ have the
same column degrees, and $\overline{\overline{G(s)}} =G(s)$.  From
the definition of $\overline{G(s)}$, we have that
$G_{\infty}=\overline{G} _0$; since $G(s)$ is minimal, $\overline{G}
_0$ also has full rank.

Suppose $p(s)\in \F [s]$ has degree $d$ and is a common divisor of
the $k\times k$ minors of $\overline{G(s)}$.  Since $\overline{G}
_0$ has full rank, $p(0)\neq 0$, so that $s^dp(s^{-1})$ has degree
$d$.  Since $\overline{\overline{G(s)}} =G(s)$, $s^dp(s^{-1})$ is a
common divisor of the $k\times k$ minors of $G(s)$. As $\mathcal{C}$
is a summand of $\F [s]^n$, it follows that $d=0$, so that the only
common divisors of the $k\times k$ minors of $\overline{G(s)}$ are
the nonzero elements of $\F$. Thus, the column space of
$\overline{G(s)}$ is a summand of $\F [s]^n$, which means that it is
a rate $k/n$ convolutional code $\overline{\mathcal{C}}$. It follows
from the remarks in the preceding paragraph that $\overline{G(s)}$
is a minimal generator matrix of $\overline{\mathcal{C}}$ and that
$\overline{\mathcal{C}}$ has degree $\delta$.

Consider the vector polynomials
$$
v(s):=v_{d}+v_{d -1}s+\cdots +v_{1}s^{d -1}+v_{0}s^{d}
$$
and
$$
\overline{v(s)}:=v_{0}+v_{1}s+\cdots +v_{d -1}s^{d -1}+v_{d}s^{d}
$$
in $\F ^n[s]$, and note that $\overline{v(s)} =s^{d}v(s^{-1})$.
Thinking of $v(s)$ and $\overline{v(s)}$ as column vectors in $\F
[s]^n$, we observe that a $(k+1)\times (k+1)$ minor of
$\begin{bmatrix} G(s)&\vline&v(s)\end{bmatrix}$ is zero if and only
if the corresponding minor of $\begin{bmatrix}\overline{G(s)}&\vline
&\overline{v(s)}
\end{bmatrix}$ is.  Since $\mathcal{C}$ and $\overline{\mathcal{C}}$
are summands of $\F [s]^n$, this means that $v(s)\in \mathcal{C}$ if
and only if $\overline{v(s)}\in \overline{\mathcal{C}}$.
\end{proof}

This result, together with the remarks preceding it, shows that
$$
\binom{y_{0}}{u_{0}},\binom{y_{1}}{u_{1}},\ldots ,\binom{y_{d
-1}}{u_{d -1}},\binom{y_{d}}{u_{d}}
$$
is a finite-weight sequence for $(A,B,C,D)$ if and only if
$$
\binom{y_{0}}{u_{0}} +\binom{y_{1}}{u_{1}} s + \cdots +
 \binom{y_{d -1}}{u_{d -1}} s^{d
-1} +\binom{y_{d}}{u_{d}} s^{d} \in\overline{\mathcal{C}}.
$$
$\overline{\mathcal{C}}$, then, will have the property that
$d_M^c(\overline{\mathcal{C}})=(n-k)\Big
(\Big\lfloor\frac{\delta}{k} \Big\rfloor +1\Big )+\delta +1$.  For
the rest of the paper, we will refer to the code
$\overline{\mathcal{C}}$ as the code represented by the matrices
$(A,B,C,D)$.

\section{Trivial Rank Deficiency and the sMDS Property}
In this section, we give conditions on the entries of the matrices
in a realization $(A,B,C,D)\in S ^{\delta}_{k,n}$ guaranteeing that
the convolutional code these matrices represent has both the MDP and
sMDS properties. For $(A,B,C,D)\in S ^{\delta}_{k,n}$ and $j\in \N
_{0}$, we form the matrices
\begin{equation}                          \label{Bl-To}
       \mathcal{T}_j  :=
 \left[
  \begin{array}{ccccc}
  D           &        0           &\cdots    &\cdots &0 \\
  CB          &        D           &\ddots    &       &\vdots  \\
  CAB         &        CB          &\ddots    &\ddots &\vdots  \\
  \vdots      &        \vdots      &\ddots    &\ddots &0  \\
  CA^{j -1}B  &        CA^{j -2}B  &\cdots    & CB    & D
  \end{array}
  \right].
\end{equation}
\begin{notation}
Let $l_1,l_2 \in \N$ satisfy $1\leq l_1\leq (j+1)(n-k)$ and $1\leq
l_2\leq (j+1)k$. Let $1\leq i_1<\cdots <i_{l_1}\leq (j+1)(n-k)$ and
$1\leq j_1<\cdots <j_{l_2}\leq (j+1)k$ be two sequences of integers.
We denote by $(\mathcal{T} _j)_{j_1,\ldots ,j_{l_2}}^{i_1,\ldots
,i_{l_1}}$ the $l_1\times l_2$ submatrix obtained from $\mathcal{T}
_j$ by intersecting rows $i_1,\ldots ,i_{l_1}$ and columns
$j_1,\ldots ,j_{l_2}$.
\end{notation}

\noindent Notice that, if $(\mathcal{T} _j)_{j_1}^{i_1}\not =0$,
then $j_1\leq\lceil\frac{i_1}{n-k}\rceil k$.

In what follows, the notion of trivial rank deficiency plays an
important role. To define trivial rank deficiency, we think of
replacing the entries of the block matrices in $\mathcal{T}_j$ with
the indeterminates of the polynomial ring
$R:=\K[x_1,x_2,\ldots,x_{(j+1)(n-k)k}]$. Specifically, we replace
the entry $(s,t)$ of the matrix $D$ with the indeterminate $x_{(s -
1)k + t}$ and the entry $(s,t)$ of the matrix $CA^{i}B$ with the
indeterminate $x_{(i+1)(n - k)k + (s - 1)k + t}$.  The zero entries
above the block diagonal remain zero.

\begin{defi}\label{trd}
Let $c$ be an integer with $0\leq c\leq n-k-1$, and let $l$ be an
integer satisfying $1\leq l\leq \min \{ (j+1)(n-k)-c,(j+1)k\}$.  A
square submatrix of $\mathcal{T}_j$ is said to be {\em trivially
rank deficient} if the determinant of this submatrix is zero when it
is viewed as a matrix over $R$ in the manner described above. A
submatrix
$(\mathcal{T}_j)^{i_1,i_2,\ldots,i_{l+c}}_{j_1,j_2,\ldots,j_l}$ of
$\mathcal{T}_j$ is called {\em trivially rank deficient} if all
${l+c \choose l}$ $l\times l$ submatrices of
$(\mathcal{T}_j)^{i_1,i_2,\ldots ,i_{l+c}} _{j_1,j_2,\ldots,j_l}$
are trivially rank deficient.
\end{defi}
To say that $(\mathcal{T}_j)^{i_1,i_2,\ldots ,i_{l+c}}
_{j_1,j_2,\ldots,j_l}$ is trivially rank deficient is to say that
$(\mathcal{T}_j)^{i_1,i_2,\ldots ,i_{l+c}} _{j_1,j_2,\ldots,j_l}$
has less than full rank regardless of how elements of $\K$ are
substituted for the indeterminates of $R$.  The next lemma shows how
to determine if a given submatrix is trivially rank deficient.

\begin{lemma}\label{B}
Let $l$ be an integer with $1\leq l\leq \min \{
(j+1)(n-k)-c,(j+1)k\}$, and let $(\mathcal{T}_j)^{i_1,i_2,\ldots
,i_{l+c}}_{j_1,j_2,\ldots ,j_l}$ be an $(l+c)\times l$ submatrix of
$\mathcal{T}_j$.  Then, the following are equivalent:
\begin{enumerate}
\item $(\mathcal{T}_j)^{i_1,i_2,\ldots ,i_{l+c}}_{j_1,j_2,\ldots ,j_l}$ is trivially
rank deficient.
\item $(\mathcal{T}_j)^{i_{1+c},i_{2+c},\ldots ,i_{l+c}}_{j_1,j_2,\ldots ,j_l}$ is trivially rank deficient.
\item The inequality
$$
   j_t>\Big\lceil\frac{i_{t+c}}{n-k}\Big\rceil k
$$
holds for some $t\in\{ 1,\ldots ,l \}$.
\end{enumerate}
\end{lemma}
\begin{proof}
For notational convenience, we set $(\mathcal{T} _j)_{\bar j}^{\bar
i}:=(\mathcal{T}_j)^{i_{1+c},i_{2+c},\ldots
,i_{l+c}}_{j_1,j_2,\ldots ,j_l}$ and $(\mathcal{T} _j)^{\tilde
i}_{\bar j}:=(\mathcal{T}_j)^{i_1,i_2,\ldots
,i_{l+c}}_{j_1,j_2,\ldots ,j_l}$.\\
${\bf 1 \Longrightarrow 2}$:  Suppose $(\mathcal{T} _j)^{\tilde
i}_{\bar j}$ is trivially rank deficient. Then, by definition, all
${l+c \choose l}$ $l\times l$ submatrices of $(\mathcal{T}
_j)^{\tilde i}_{\bar j}$ are trivially rank deficient.  In
particular, $(\mathcal{T} _j)_{\bar j}^{\bar i}$ is trivially rank deficient.\\
${\bf 2 \Longrightarrow 3}$:  Suppose that $(\mathcal{T} _j)_{\bar
j}^{\bar i}$ is trivially rank deficient.  We first use induction on
$l$ to prove that $(\mathcal{T} _j)_{\bar j}^{\bar i}$ is lower
block triangular and has a 0 on its diagonal. If $l=1$ and
$(\mathcal{T} _j)_{\bar j}^{\bar i}$ is trivially rank deficient,
then the claim is trivially true. Suppose $l$ satisfies $2\leq l
\leq \min \{ (j+1)(n-k)-c,(j+1)k\}$, that the induction hypothesis
is satisfied for $1,2,\ldots ,l-1$, and that $(\mathcal{T} _j)_{\bar
j}^{\bar i}$ is trivially rank deficient. If $(\mathcal{T}
_j)_{j_1}^{i_{l+c}} =0$, then every entry in $(\mathcal{T} _j)_{\bar
j}^{\bar i}$ is 0, and thus all diagonal entries are 0.  If
$(\mathcal{T} _j)_{j_1}^{i_{l+c}}\neq 0$, then let $x_{\iota}$ be
the indeterminate corresponding with $(\mathcal{T}
_j)_{j_1}^{i_{l+c}}$ when $\mathcal{T} _j$ is viewed over $R$ in the
manner described before Definition \ref{trd}. Notice that, when
$(\mathcal{T} _j)_{\bar j}^{\bar i}$ is viewed in this way, the
indeterminate $x_{\iota}$ appears exactly once. Since $x_{\iota}$ is
transcendental over $\K (x_1,\ldots ,x_{\iota -1})$, doing a
cofactor expansion along the first column of $(\mathcal{T} _j)_{\bar
j}^{\bar i}$ (still viewing $(\mathcal{T} _j)_{\bar j}^{\bar i}$
over $R$) shows that the $(l - 1)\times (l - 1)$ submatrix
$(\mathcal{T}_j)^{i_{1+c},i_{2+c},\ldots ,i_{l-1+c}}_{j_2,j_3,\ldots
,j_l}$ is trivially rank deficient.  By the induction hypothesis,
$(\mathcal{T}_j)^{i_{1+c},i_{2+c},\ldots ,i_{l-1+c}}_{j_2,j_3,\ldots
,j_l}$ is lower block triangular and has a 0 on its diagonal. It
follows that there is an integer $h$ satisfying $1\leq h\leq l-1$
such that $(\mathcal{T} _j)_{j_{h+1}}^{i_{h+c}}= 0$. This, in turn,
means that $(\mathcal{T} _j)_{\bar j}^{\bar i}$ is lower block
triangular. Because we assumed that $(\mathcal{T} _j)_{\bar j}^{\bar
i}$ is trivially rank deficient, it follows that at least one of
$(\mathcal{T}_j)^{i_{1+c},i_{2+c},\ldots ,i_{h+c}}_{j_1,j_2,\ldots
,j_h}$ and $(\mathcal{T}_j)^{i_{h+1+c},i_{h+2+c},\ldots
,i_{l+c}}_{j_{h+1},j_{h+2},\ldots ,j_l}$ is trivially rank
deficient. By the induction hypothesis, at least one of these
submatrices is lower block triangular and has a 0 on its diagonal.
As the diagonals of these submatrices lie on the diagonal of
$(\mathcal{T} _j)_{\bar j}^{\bar i}$, the claim follows.

Next, we note that the diagonal entries of $(\mathcal{T} _j)_{\bar
j}^{\bar i}$ are the entries $(\mathcal{T} _j)_{j_1}^{i_{1+c}}$,
$(\mathcal{T} _j)_{j_2}^{i_{2+c}}$, $\ldots$ , $(\mathcal{T}
_j)_{j_l}^{i_{l+c}}$. From the structure of $\mathcal{T} _j$, it is
clear that, when $(\mathcal{T} _j)_{\bar j}^{\bar i}$ is viewed over
$R$, a diagonal entry $(\mathcal{T} _j)_{j_t}^{i_{t+c}}$ is 0 if and
only if
$$
j_t > \Big\lceil\frac{i_{t+c}}{n-k}\Big\rceil k.
$$
It follows that
$$
j_t > \Big\lceil\frac{i_{t+c}}{n-k}\Big\rceil k
$$
for some $t\in \{1,\ldots ,l\}$.
\\
${\bf 3 \Longrightarrow 1}$:  If
$$
   j_t>\Big\lceil\frac{i_{t+c}}{n-k}\Big\rceil k
$$
for some $t\in\{ 1,\ldots ,l \}$, then $(\mathcal{T} _j)_{\bar
j}^{\bar i}$ has a 0 on its diagonal and is lower block triangular.
$(\mathcal{T} _j)_{\bar j}^{\bar i}$ is therefore trivially rank
deficient. Let $(\mathcal{T}_j)^{w_1,w_2,\ldots
,w_l}_{j_1,j_2,\ldots ,j_l}$ be a submatrix of $(\mathcal{T}
_j)^{\tilde i}_{\bar j}$.  Since $w_t\leq i_{t+c}$,
$$
   j_t>\Big\lceil\frac{w_t}{n-k}\Big\rceil k
$$
holds as well.  As before, it follows that
$(\mathcal{T}_j)^{w_1,w_2,\ldots ,w_l}_{j_1,j_2,\ldots ,j_l}$ is
trivially rank deficient. Consequently, all $l+c\choose l$ $l\times
l$ submatrices of $(\mathcal{T} _j)^{\tilde i}_{\bar j}$ are
trivially rank deficient, so that $(\mathcal{T} _j)^{\tilde i}_{\bar
j}$ is trivially rank deficient.
\end{proof}

We next characterize the MDP and sMDS properties in terms of trivial
rank deficiency.  We denote by $r$ the difference of the generalized
Singleton bound and the upper bound for the $L$th column distance:
\begin{multline*}
   r:=(n-k)\Big (\Big \lfloor \frac{\delta}{k} \Big \rfloor +1\Big ) +\delta
   +1-\Big (\Big \lfloor
   \frac{\delta}{k} \Big \rfloor +\Big \lfloor \frac{\delta}{n-k} \Big \rfloor +1\Big
   )(n-k)-1=\\ \delta -\Big \lfloor
   \frac{\delta}{n-k} \Big \rfloor (n-k).
\end{multline*}
Note that $r$ is the remainder of $\delta$ on division by $n-k$.  If
$r=0$, then $L=M$, and a code is MDP if and only if it is sMDS. This
case was considered in~\cite{gl03r} and~\cite{12}, so we will assume
that $r\in \{1,\ldots ,n-k-1\}$.  In this situation, $M=L+1$.

\begin{theo}\label{CharsMDS}
Let $(A,B,C,D)\in S ^{\delta}_{k,n}$ and $\C$ be the
$(n,k,\delta)$-code represented by $(A,B,C,D)$. Then, $\mathcal{C}$
is an MDP code if and only if every square submatrix of $\mathcal{T}
_L$ that is not trivially rank deficient has full rank.  $\C$ is an
sMDS code if and only if, for every integer $l$ satisfying $1\leq
l\leq \min\{(M+1)(n-k)-(n-k-r),(M+1)k\}$, every submatrix
$(\mathcal{T}_M)^{i_1,i_2,\ldots ,i_{l+n-k-r}}_{j_1,j_2,\ldots
,j_l}$ that is not trivially rank deficient has full rank.
\end{theo}
\begin{proof}
The first statement is~\cite[Corollary 2.5]{12}.  We consider next
the second statement.\\
$\Longleftarrow$: Suppose that
$$
  v:=\left[
  \begin{array}{ccccccccc}
  y_0^T &  y_1^T &\ldots & y_M^T   &  \vline&u_0^T&u_1^T & \cdots  &
  u_M^T
\end{array}\right]^T
$$
is formed from the first $M+1$ vectors of a finite-weight sequence
for $(A,B,C,D)$ with $u_0\not = 0$, so that the matrix equation
\begin{equation*}
\left[
\begin{tabular}{lllccccc}
\multicolumn{3}{c|}{} & $D$ &$0$  &$\cdots$  &$\cdots$  &$0$  \\
\multicolumn{3}{c|}{} & $CB$ & $D$ &$\ddots$  &  &$\vdots$  \\
\multicolumn{3}{c|}{$-I_{(M +1)(n-k)}$} & $CAB$ & $CB$  &$\ddots$  &$\ddots$  &$\vdots$  \\
\multicolumn{3}{c|}{} & $\vdots $ &$\vdots$  & $\ddots$ & $\ddots$ &$0$  \\
&&\multicolumn{1}{l|}{} & $CA^{M -1}B$ & $CA^{M -2}B$ & $ \cdots $ &
$CB$ & $D$
\end{tabular}
\right] v =0_{(M +1)n}
\end{equation*}
is satisfied, and denote the weight of
$$
u:=\left[
  \begin{array}{cccc}
  u_0^T & u_1^T &\cdots &  u_M^T
  \end{array}
\right]^T
$$
by $w$.  For $t\in \{1,\ldots ,w \}$, let $j_t$ denote the position
of the $t$th nonzero entry in $u$, and let $\bar u$ denote the
vector obtained from $u$ by deleting all of the zero entries.
Suppose that $(\mathcal{T} _M)_{\bar j}^{\tilde
i}:=(\mathcal{T}_M)^{i_1,i_2,\ldots ,i_{w+n-k-r}}_{j_1,j_2,\ldots
,j_w}$ is a submatrix of $\mathcal{T}_M$ such that
\begin{equation}\label{equat}
(\mathcal{T} _M)_{\bar j}^{\tilde i}\bar u=0.
\end{equation}
Since $u_0 \not = 0$, we have that
$$
   j_1\leq k \leq \Big \lceil \frac{i_{1}}{n-k} \Big \rceil k\leq \Big \lceil \frac{i_{1+n-k-r}}{n-k} \Big \rceil
   k.
$$
By Lemma \ref{B}, $(\mathcal{T}_M)^{i_1,i_2,\ldots
,i_{1+n-k-r}}_{j_1}$ is not trivially rank deficient, so that it has
full rank.  This means that at least one of its entries is nonzero,
and, since (\ref{equat}) holds, it follows that
$$
   j_2\leq \Big \lceil \frac{i_{1+n-k-r}}{n-k} \Big \rceil k\leq \Big \lceil \frac{i_{2+n-k-r}}{n-k} \Big \rceil
   k.
$$
By Lemma \ref{B}, $(\mathcal{T}_M)^{i_1,i_2,\ldots
,i_{2+n-k-r}}_{j_1,j_2}$ is not trivially rank deficient, so that it
has full rank.  Consequently, at least one $2\times 2$ minor of
$(\mathcal{T}_M)^{i_1,i_2,\ldots ,i_{2+n-k-r}}_{j_1,j_2}$ is
nonzero.  Again, since (\ref{equat}) holds, it follows that
$$
   j_3\leq \Big \lceil \frac{i_{2+n-k-r}}{n-k} \Big \rceil k\leq \Big \lceil \frac{i_{3+n-k-r}}{n-k} \Big \rceil
   k.
$$
Continuing, we see that, for $t\in \{ 1,\ldots ,w \}$,
$$
   j_t\leq \Big \lceil \frac{i_{t+n-k-r}}{n-k} \Big \rceil k.
$$
A final application of Lemma~\ref{B} gives that $(\mathcal{T}
_M)_{\bar j}^{\tilde i}$ is not trivially rank deficient. By
hypothesis, it must have full rank, which contradicts the hypothesis
that $(\mathcal{T} _M)_{\bar j}^{\tilde i}\bar u=0$. Consequently,
at most $w+n-k-r-1$ rows of $(\mathcal{T} _M)_{\bar j}^{\tilde i}$
are in the left kernel of $\bar u$.  It follows that $v$ has weight
at least $w + ((M + 1)(n - k) - (w + n - k - r - 1)) = M(n - k) + 1
+ r = (L + 1)(n - k) + 1 + r$, which means that $d_M^c(\mathcal{C})
\geq (L + 1)(n - k) + 1 + r$. Recalling Proposition
\ref{P-dcj.bound} and the definition of $r$, we conclude that
$d_M^c(\mathcal{C})= (L + 1)(n - k) + 1 + r$, so that $\mathcal{C}$
is sMDS.

$\Longrightarrow:$ We prove the contrapositive.  Suppose that the
matrix\eqr{Bl-To} has a $(w + n - k - r)\times w$ submatrix
$(\mathcal{T} _M)_{\bar j}^{\tilde
i}:=(\mathcal{T}_M)^{i_1,i_2,\ldots ,i_{w+n-k-r}}_{j_1,j_2,\ldots
,j_w}$ that is not trivially rank deficient and that has less than
full rank. There then exists a vector $(\mathcal{T} _M)_{\bar
j}^{\tilde i}\bar u\neq 0$ of weight $w'\leq w$ such that $\bar
u=0$.  Let
$$
u:=\left[
  \begin{array}{cccc}
  u_0^T & u_1^T &\cdots &  u_M^T
  \end{array}
\right]^T\in \F ^{Mk}
$$
be the vector in which the $j_t$th entry is the $t$th entry of $\bar
u$ and all other entries are zero; because of the block Toeplitz
structure of $\mathcal{T} _M$, we may assume that $u_0\neq 0$. Using
(\ref{iso}), we form the vector
$$
v:=\left[
  \begin{array}{ccccccccc}
  y_0^T &  y_1^T &  \cdots &  y_M^T  &   \vline&  u_0^T&
  u_1^T &
  \cdots &   u_M^T
  \end{array}
\right]^T.
$$
Because $[-I_{(M+1)(n-k)}\mid \mathcal{T} _M]v=0$, the weight of $v$
is at most $w' + (M + 1)(n - k) - (w - n - k - r)\leq (M + 1)(n - k)
- (n - k - r) = (L + 1)(n - k) + r<(L+1)(n-k)+r+1$. We may choose
additional information vectors $u_{M+1},\ldots ,u_{d}$ so that
$x_{d}=0$ (see, for example,~\cite{Antsaklis}); in other words, it
is possible to extend $v$ into a finite-weight sequence for
$(A,B,C,D)$ with weight less than the generalized Singleton bound.
Thus, $d_M^c(\mathcal{C})<(L+1)(n-k)+r+1$, so that $\mathcal{C}$ is
not an sMDS code.
\end{proof}

Theorem \ref{CharsMDS} gives polynomial conditions on the entries of
a realization $(A,B,C,D)\in S ^{\delta}_{k,n}$ that may be used to
determine whether or not the convolutional code these matrices
represent has the MDP and sMDS properties.  In the next section, we
use this information to show that we can find a realization
$(A,B,C,D)\in S ^{\delta}_{k,n}$ representing an $(n,k,\delta
)$-code that has the MDP and sMDS properties.

\section{Proof of the Existence of sMDS Convolutional Codes}
Recall that we defined the block matrices making up $\mathcal{T} _M$
in terms of matrices $(A,B,C,D)$.  In this section, we will work in
the opposite direction.  Let $\{ F_0,F_1,\ldots ,F_j \}$ be a
sequence of matrices in $\K ^{(n-k)\times k}$.  Slightly abusing
notation, we set
\begin{equation}                          \label{Bl-To2}
       \mathcal{T}_j  :=
 \left[
  \begin{array}{cccc}
  F_0           &        0           &\cdots    &0 \\
  F_1          &        F_0           &\ddots    &\vdots  \\
  \vdots      &        \vdots      &\ddots    &0  \\
  F_j  &        F_{j-1}  &\cdots    & F_0
  \end{array}
  \right].
\end{equation}
The plan is to show the existence of a sequence $\{ F_0,F_1,\ldots
,F_M \}$ of matrices in $\K ^{(n-k)\times k}$ such that
\begin{enumerate}
\item $\mathcal{T} _M$
has the property that, for all integers $l$ with $1\leq l\leq \min
\{(M+1)(n-k)-(n-k-r),(M+1)k \}$, every submatrix
$(\mathcal{T}_M)^{i_1,i_2,\ldots ,i_{l+n-k-r}}_{j_1,j_2,\ldots
,j_l}$ that is not trivially rank deficient has full rank;
\item there is a minimal partial realization
$(A,B,C,D)\in S ^{\delta}_{k,n}$ of this matrix sequence (this means
that $D=F_0$ and $CA^{i-1}B=F_i$ for $1\leq i\leq M$).
\end{enumerate}
The matrices $(A,B,C,D)$ will represent the desired code. We begin
with the following lemma.
\begin{lemma}\label{Ex1}
There exists a sequence $\{ F_0,F_1,\ldots ,F_L \}$ of matrices in
$\K ^{(n-k)\times k}$ such that every square submatrix of
$\mathcal{T} _L$ that is not trivially rank deficient has full rank.
\end{lemma}
\begin{proof}
  Note that we may think of such a matrix sequence $\{ F_0,F_1,\ldots ,F_L \}$ as a point in $\K ^{(L+1)(n-k)k}$.  To begin, think of the matrix (\ref{Bl-To2}) with $j=L$ as being defined over the polynomial ring $\K
  [x_1,x_2,\ldots,x_{(L+1)(n-k)k}]$, the entries corresponding with the indeterminates of this ring in a manner
  analogous to that in the previous section.  When viewed in this way, the determinant of a square submatrix of
  $\mathcal{T} _L$ that is not trivially rank deficient is a nonzero polynomial in $\K
  [x_1,x_2,\ldots,x_{(L+1)(n-k)k}]$, and there is a finite number of such polynomials.  The solution sets of these polynomials
  make up a proper algebraic subset of $\K ^{(L+1)(n-k)k}$, the complement of which is a nonempty Zariski open set.  Choose
  $\left \{ F_0,\ldots ,F_L \right \}$ to be a point in this open set.
\end{proof}

To determine the degree of a minimal partial realization of a matrix
sequence $\{ F_0,F_1,\ldots ,F_M \}$, we consider the matrices
$$
   \mathcal{F}_{x,y}:= \left[
      \begin{array}{cccc}
                      F_1    & F_2    & \cdots & F_y      \\
                      F_2    & F_3    & \cdots & F_{y+1}  \\
                      \vdots & \vdots &        & \vdots   \\
                      F_x    & F_{x+1}& \cdots & F_{x+y-1}
      \end{array}
   \right].
$$
In~\cite[Lemma 3]{te70}, it is shown that the degree of a minimal
partial realization of $\{ F_0,F_1,\ldots ,F_M \}$ is given by the
expression
\begin{equation}\label{rank}
\sum_{x =1}^{M} \text{rank } \mathcal{F} _{x ,M+1-x} - \sum_{x
=1}^{M-1} \text{rank } \mathcal{F} _{x ,M-x}.
\end{equation}
The next results show that, starting with a matrix sequence $\{
F_0,F_1,\ldots ,F_L \}$ as described in Lemma \ref{Ex1}, we can find
a matrix $F_M$ so that the expression (\ref{rank}) evaluates to
$\delta$.

\begin{lemma}\label{FM1}
Let $\left\{F_0,\ldots,F_M\right\}$ be a sequence of matrices in $\K
^{(n-k)\times k}$ such that every square submatrix of $
\mathcal{T}_L$ that is not trivially rank deficient has full rank.
Then,
\begin{enumerate}
\item For $x\in \{1,\ldots ,M-1 \}$, $\rank \mathcal{F} _{x ,M-x }=\min \{ x
(n-k),(M-x )k \}$.
\item If $\rank \mathcal{F}_{x
,M+1-x} <\min \{ x (n-k),(M+1-x )k \}$, then $x=\lceil M\frac{k}{n}
\rceil$. If $x\in \{1,\ldots ,M \} \backslash \{ \lceil M\frac{k}{n}
\rceil \}$, then $\rank \mathcal{F} _{x ,M+1-x }=\min \{ x
(n-k),(M+1-x ) k \}$.
\item Set $\bar x:=\lceil M\frac{k}{n} \rceil$. The expression
(\ref{rank}) reduces to $\rank \mathcal{F} _{\bar x,M+1-\bar x}$.
\end{enumerate}
\end{lemma}
\begin{proof}
To verify the first claim, observe that $\mathcal{F} _{x ,M-x }$
differs by a column permutation from a submatrix of $\mathcal{T} _L$
that has full rank.

For the second claim, suppose first that $x(n-k)\leq (M+1-x)k$. The
hypothesis is then that $\rank \mathcal{F}_{x ,M+1-x } <x (n-k)$. If
$x <M$, it follows from 1 that $\rank \mathcal{F}_{x , M-x } =\min
\{ x (n-k),(M-x )k \}$, which means that $x (n-k)>(M-x )k$.
Together, this gives
$$
   (M-x )k<x (n-k)\leq (M+1-x )k
$$
(note that the first inequality also holds if $x =M$). This can be
rewritten as
$$
   M\frac{k}{n} <x \leq (M+1)\frac{k}{n}.
$$
If we suppose instead that $(M+1-x )k\leq x (n-k)$, similar
reasoning leads to
$$
   (M+1)\frac{k}{n} \leq x <M\frac{k}{n} +1.
$$
In all, we have
$$
  M\frac{k}{n} <x <M\frac{k}{n} +1.
$$
Since $x$ is an integer, $x=\lceil M\frac{k}{n} \rceil$.  The second
statement follows immediately.

The third claim follows directly from the first two, since $x<\bar x
\implies x(n-k)<(M-x)k$ and $x>\bar x \implies x(n-k)>(M+1-x)k$.
\end{proof}

\begin{theo}\label{FM3}
Let $\left\{F_0,\ldots ,F_L\right\}$ be a sequence of matrices in
$\K ^{(n-k)\times k}$ such that every square submatrix of
$\mathcal{T}_L $ that is not trivially rank deficient has full rank.
Then, one can find a matrix $F _M\in \K ^{(n-k)\times k}$ such that
\begin{enumerate}
\item the matrix
$$
   \mathcal{F}_{\bar x,M+1-\bar x}= \left[
      \begin{array}{cccc}
                      F_1           & F_2         & \cdots & F_{M+1-\bar x}  \\
                      F_2           & F_3         & \cdots & F_{M+2-\bar x}  \\
                      \vdots        & \vdots      &        & \vdots          \\
                      F_{\bar x}    & F_{\bar x +1}& \cdots & F_M
      \end{array}
   \right]
$$
has rank $\delta$.
\item the matrix $\mathcal{T}_M$ has the property that, for every
integer $l$ with $1\leq l\leq \min \{(M+1)(n-k)-(n-k-r),(M+1)k \}$,
every submatrix $(\mathcal{T}_M)^{i_1,i_2,\ldots
,i_{l+n-k-r}}_{j_1,j_2,\ldots ,j_l}$ that is not trivially rank
deficient has full rank.
\end{enumerate}
\end{theo}
\begin{proof}
We may write
$$
\delta =\Big\lfloor \frac{\delta}{n-k} \Big\rfloor
(n-k)+r=\Big\lfloor \frac{\delta}{k} \Big\rfloor k+r',
$$
where $1\leq r<n-k$ and $0\leq r'<k$.  Since
$$
M=L+1=\Big\lfloor \frac{\delta}{n-k} \Big\rfloor +\Big\lfloor
\frac{\delta}{k} \Big\rfloor +1,
$$
we see that
\begin{align*}
\frac{Mk}{n}&=\Big\lfloor \frac{\delta}{n-k} \Big\rfloor \frac{k}{n}
+\Big\lfloor \frac{\delta}{k} \Big\rfloor \frac{k}{n} +\frac{k}{n}
=\Big\lfloor \frac{\delta}{n-k} \Big\rfloor -\Big\lfloor
\frac{\delta}{n-k} \Big\rfloor \frac{n-k}{n} +\Big\lfloor
\frac{\delta}{k} \Big\rfloor \frac{k}{n} +\frac{k}{n}\\
&=\Big\lfloor \frac{\delta}{n-k} \Big\rfloor -\frac{\delta -r}{n}
+\frac{\delta -r'}{n} +\frac{k}{n} =\Big\lfloor \frac{\delta}{n-k}
\Big\rfloor +\frac{k-r'+r}{n}.
\end{align*}
Since $1<k-r'+r<n$, we have $\lfloor \frac{\delta}{n-k} \rfloor
=\bar x -1$, so that $\delta =(\bar x -1)(n-k)+r$ and
$$
\frac{Mk}{n} =\bar x -1+\frac{k-r'+r}{n}.
$$
Multiplying both sides by $n$ and subtracting $\bar x k$ from both
sides, we get
$$
(M-\bar x )k=(\bar x -1)(n-k)+r-r'=\delta -r',
$$
from which it follows that $(M-\bar x )k\leq \delta$.  Since $r'
<k$, it also follows that $\delta <(M+1-\bar x )k$.

We next want to see that we may find a matrix $F_M$ as described in
the statement of the theorem.  We first consider the top $r$ rows of
$F_M$.  Using the same reasoning as in the proof of Lemma \ref{Ex1},
we may find elements of $\K$ to form these top $r$ rows so that all
square submatrices of the top $M(n-k)+r$ rows of $\mathcal{T} _M$
that are not trivially rank deficient have full rank.  In
particular, all square submatrices of the top $\delta$ rows of
$\mathcal{F}_{\bar x,M+1-\bar x}$ have full rank.  Denote the
$r\times k$ matrix consisting of these $r$ rows by $F_M'$.  Since
$\delta <(M+1-\bar x )k$, $\text{rank } \mathcal{F} _{\bar x
,M+1-\bar x} \geq
 \delta$ will hold regardless of how the entries of the bottom $n-k-r$ rows of $F_M$ are chosen.  To find entries for these rows
 so that $\rank \mathcal{F}_{\bar x,M+1-\bar x} =\delta$, consider the top
$\delta$ rows of $\mathcal{F} _{\bar x ,M-\bar x}$. Since $\delta
\geq (M-\bar x )k$, we may choose $M-\bar x$ of these $\delta$ rows
to form an $(M-\bar x )k\times (M-\bar x )k$ submatrix that
necessarily has full rank. This means that the last $n-k-r$ rows of
$\mathcal{F} _{\bar x ,M-\bar x }$ may each be expressed as a linear
combination of the rows of our chosen submatrix. Consequently, we
may take the last $n-k-r$ rows of $F_M$ to be the corresponding
linear combinations of the rows of
$$
   \left[
      \begin{array}{c}
                      F_{M+1-\bar x}  \\
                      F_{M+2-\bar x}  \\
                      \vdots          \\
                      F_M'
      \end{array}
   \right]
$$
extending the rows of our chosen submatrix.  With this, we have
found an $F_M$ so that $\text{rank } \mathcal{F} _{\bar x ,M+1-\bar
x } =\delta$.

Suppose finally that $(\mathcal{T} _M)_{\bar j}^{\bar
i}:=(\mathcal{T}_M)^{i_1,i_2,\ldots ,i_{l+n-k-r}}_{j_1,j_2,\ldots
,j_l}$ is a submatrix of $\mathcal{T}_M$ that is not trivially rank
deficient and does not have full rank.  Then, in particular,
$(\mathcal{T}_M)^{i_1,i_2,\ldots ,i_l}_{j_1,j_2,\ldots ,j_l}$ does
not have full rank.  Since $(\mathcal{T}_M)^{i_1,i_2,\ldots
,i_l}_{j_1,j_2,\ldots ,j_l}$ is contained in the top $M(n-k)+r$ rows
of $\mathcal{T} _M$, it must be trivially rank deficient. By Lemma
\ref{B}, there exists a smallest integer $t\in \{ 1,\ldots l \}$
such that
$$
 j_t>\Big\lceil \frac{i_t}{n-k} \Big\rceil k.
$$
Since $(\mathcal{T} _M)_{\bar j}^{\bar i}$ is not trivially rank
deficient, it also follows from Lemma \ref{B} that
$$
   j_{\tau}\leq\Big\lceil\frac{i_{\tau +n-k-r}}{n-k}\Big\rceil k \,\,\, \forall \,\tau \in
   \{1,\ldots ,l \},
$$
so that $(\mathcal{T} _M)_{\bar j}^{\tilde i}:=(\mathcal{T}
_M)^{i_{t+n-k-r},i_{t+1+n-k-r},\ldots
,i_{l+n-k-r}}_{j_t,j_{t+1},\ldots ,j_l}$ is not trivially rank
deficient.  Since $j_t>k$, $(\mathcal{T} _M)_{\bar j}^{\tilde i}$
must be a submatrix of $\mathcal{T} _L$.  Thus, $(\mathcal{T}
_M)_{\bar j}^{\tilde i}$ has full rank.  Recalling how $t$ was
chosen, we conclude that $(\mathcal{T} _M)_{\bar j}^{\bar i}$ has
full rank. This is a contradiction. We conclude that, if a submatrix
$(\mathcal{T} _M)_{\bar j}^{\bar i}$ is not trivially rank
deficient, then it has full rank.
\end{proof}

\begin{cor}\label{FM4}
Let $\left\{F_0,\ldots ,F_M\right\}$ be as in Theorem~\ref{FM3}.
Then, the expression (\ref{rank}) evaluates to $\delta$.
\end{cor}

We are now ready to finish our existence proof.

\begin{theo}\label{Ex4}
An MDP and sMDS $(n,k,\delta )$-code exists over a sufficiently
large finite field.
\end{theo}
\begin{proof}
By Lemma \ref{Ex1} and Theorem~\ref{FM3}, we can find a sequence
$\left\{F_0,\ldots ,F_M\right\}$ of matrices in $\K ^{(n-k)\times
k}$ such that
\begin{enumerate}
\item every square submatrix of $ \mathcal{T}_L$ that is not
trivially rank deficient has full rank
\item every $(l+n-k-r)\times l$
submatrix of $\mathcal{T}_M$ that is not trivially rank deficient
has full rank
\item the minimum possible degree of a partial realization of
$\left\{F_0,\ldots ,F_M\right\}$ is $\delta$.
\end{enumerate}
Since there are a finite number of entries in the matrices
$\left\{F_0,\ldots ,F_M\right\}$, the entries all belong to some
finite subfield $\F$ of $\K$.  From~\cite[Theorem 1]{te70}, there is
a minimal realization $(A,B,C,D)\in S ^{\delta}_{k,n}$ of the
sequence $\left\{F_0,\ldots,F_M\right\}$ with entries in $\F$. By
Theorem~\ref{CharsMDS}, the $(n,k,\delta )$-code represented by
$(A,B,C,D)$ is both MDP and sMDS.
\end{proof}

With this, we have shown that the conjecture in~\cite{gl03r} that
codes having both the MDP and sMDS properties exist for all
parameters $(n,k,\delta )$ is correct.  It is still an open problem
as to how one may construct matrices of the form (3.4) leading to
codes with these properties, and this must be
left for future research.\\

\noindent {\bf Acknowledgements:}  The author wishes to thank
Joachim Rosenthal, Heide Gluesing-Luerssen, and Jos\'e Ignacio
Iglesias Curto for helpful comments during the preparation of this
paper.  He also wishes to thank the anonymous referees for their
careful readings and detailed comments.

\end{document}